\documentclass[openright, 12pt]{article}
\usepackage{amsmath,amsfonts}
\usepackage{amsthm}
\newtheorem{teorema}{Theorem}
\newtheorem{prop}[teorema]{Proposition}
\newtheorem{lema}[teorema]{Lemma}
\newtheorem{rem}[teorema]{Remark}

\usepackage{float} 
\usepackage{multirow} 
\usepackage[dvips]{graphicx}
\usepackage{psfrag}
\usepackage{enumerate}

\newcommand{\field}[1]{\ensuremath{\mathbb{#1}}}

\newcommand{\R}{\field{R}}

\usepackage{amssymb}
\usepackage[latin1]{inputenc}

\begin{document}
\begin{center}
\Large{On Unfoldings of Stretched Polyhedra}\\
\normalsize{G\"ozde Sert, Sergio Zamora}\\
Penn State University
\end{center}

\begin{abstract}
\noindent We give a short proof of a result obtained by Mohammad Ghomi concerning existence of nets of a convex polyhedron after a suitable linear transformation.

\end{abstract}

\section*{Introduction}

A \textit{net of a polyhedron} is an arrangement of edge-joined polygons in the plane which can be folded along its edges to become the faces of the polyhedron. The first known record of this procedure is the renaissance book \textit{Instructions for Measuring with Compass and Ruler} by Albrecht D\"{u}rer \cite{Du}. In this book, D\"{u}rer shows how to cut and develop some figures, including all five regular polyhedra.  In 1975 Geoffrey Shephard \cite{Sh} posed the problem to determine whether all convex polyhedra have a net.


\begin{figure}[h]
\centering
\includegraphics[scale=0.862]{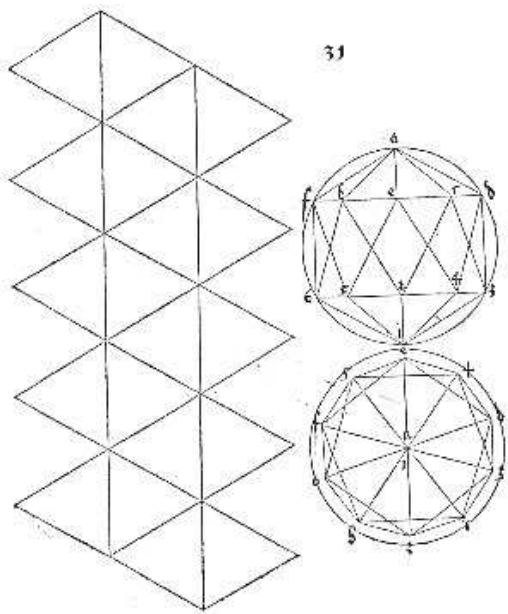}
\caption{Albrecht D\"{u}rer net of the icosahedron.}
\end{figure}

Recently Mohammad Ghomi \cite{Gh} proved the existence of a net for any polyhedron after a suitable stretching in one direction (theorem \ref{Str} below). In particular, every polyhedron is affinely equivalent to one with a net, and having a net does not depend on the combinatorial structure of the polyhedron. In this note we give an elementary proof of that result based on the arm lemma, which allows us to considerably shorten it and make it less technical.

\begin{teorema}\label{Str}
{Let $P$ be a convex polyhedron, then there is a linear stretching $L$, such that $L(P)$ has a net.}
\end{teorema}


\section*{Preliminaries}

Let $a, b, c$ be vertices of a convex polyhedron such that $ab $ and $ac$ are edges. We will denote by $\measuredangle bac $ the intrinsic angle at $a $ from $ab$ to $ac$ measured counterclockwise. For distinct $x,y,z \in \R^2$, $arg (x)$ will denote the argument from $-\pi$ to $\pi$ of $x$ as a complex number, and $\angle yxz $ will denote the angle at $x$ measured counterclockwise.

\begin{rem}
{\rm Since the polyhedron is convex, for any $a,b,c$ vertices such that $ab$ and $ac$ are edges, we have
\begin{equation}\label{Alex}
\measuredangle bac + \measuredangle  cab  < 2\pi.
\end{equation} 
}
 \end{rem}

\begin{lema} \label{Arm}
{\textbf{(Arm Lemma)} Suppose $\{ u_0, u_1, \ldots , u_m     \}, \{ v_0, v_1, \ldots , v_m      \} \subset \R^2$ satisfy
\begin{itemize}
\item $u_0 = v_0$.
\item $\vert u_j - u_{j-1} \vert = \vert v_j - v_{j-1}  \vert$ for $j=1,2, \ldots, m$.
\item $u_{j }- u_{j-1} $ and $v_j - v_{j-1} $ are almost horizontal for $j = 1, 2, \ldots , m$ (their argument is in $\left( -\frac{\pi}{10},\frac{\pi}{10} \right)$ ).
\item $arg (v_{i} - v_{i-1} )\geq arg (u_{i} - u_{i-1})$ for $i \in \{ 1, 2, \ldots, m \} $.
\end{itemize}

Then the broken lines determined by $\{ u_0, u_1, \ldots , u_m     \}$ and $ \{ v_0, v_1, \ldots , v_m      \}$ do not cross (see figure \ref{arml}) and $v_m - u_m$ is almost vertical (its argument lies in the interval $\left(  \frac{\pi}{2} - \frac{\pi}{10} , \frac{\pi}{2} + \frac{\pi}{10}  \right)$ ) 
}
\end{lema}

\begin{figure}[h]
\centering
\includegraphics[scale=0.45]{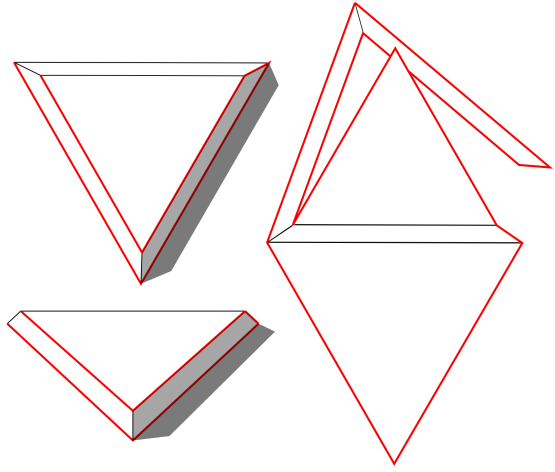}
\caption{The Arm Lemma does not hold if we remove the condition of $u_{j }- u_{j-1}$ and $ v_j - v_{j-1} $ being almost horizontal. The stretching is applied to meet this condition. }\label{Tet}
\end{figure}

\begin{figure}
\centering 
\psfrag{A}{$u_0$}
\psfrag{B}{$u_1$}
\psfrag{C}{$u_2$}
\psfrag{D}{$u_3$}
\psfrag{E}{$u_4$}
\psfrag{F}{$v_1$}
\psfrag{G}{$v_2$}
\psfrag{H}{$v_3$}
\psfrag{I}{$v_4$}
\psfrag{J}{$w_2$}
\psfrag{K}{$w_3$}
\psfrag{L}{$w_4$}
\includegraphics[scale=0.95]{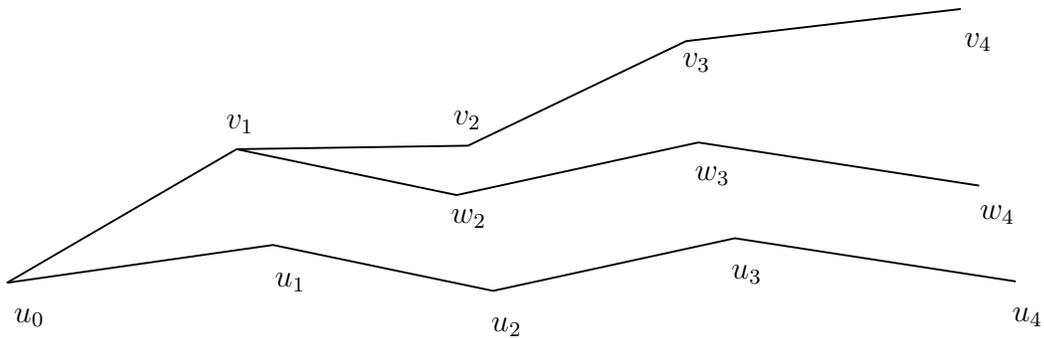}
\caption{Proof of arm lemma.} \label{arml}
\end{figure}

\begin{proof}
{\rm We will prove this lemma by induction on $m$. The case $m=1$ is elementary plane geometry.

Suppose it holds for $m=\nu$ and take $m = \nu +1$. Construct another sequence $\{ w_1, w_2, \ldots, w_m   \}$ such that $w_1 = v_1$ and $u_{j+1}u_jw_jw_{j+1} $ is a parallelogram for $j \in \{ 1, 2, \ldots , m-1\} $. Applying the induction hypothesis to $\{ w_1, w_2, \ldots w_m   \}$ and $\{v_1, v_2, \ldots, v_m  \}$, we see that they do not cross and $arg( v_{m} - w_m) \in \left(  \frac{4 \pi}{10} , \frac{6\pi}{10}  \right)$. Also, $arg(w_m - u_m )= arg (w_1 - u_1) = arg (v_1 - u_1) \in \left(  \frac{4 \pi}{10} , \frac{6\pi}{10}  \right)$, which implies $arg(v_m - u_m)\in \left(  \frac{4 \pi}{10} , \frac{6\pi}{10}  \right)$.}
\end{proof}


\begin{lema}\label{Argument}
{Let $S$ be a flat surface homeomorphic to a closed disc. Consider a map $f \colon S \rightarrow \mathbb{R}^2$ such that restricted to the interior of $S$ is an isometric immersion. Then $f$ is injective if and only if $f(\partial S)$ does not self intersect. 
}
\end{lema}

\begin{proof}
{\rm One implication is trivial. Since local isometries are conformal maps, the other one follows from the fact that for conformal maps $f \colon S \rightarrow \mathbb{R}^2$, the number of preimages $f^{-1}(x)$ equals the winding number $I(f(\partial S), x )$ for all $x \in \R^2  \backslash f(\partial S)$, which is a standard fact in complex analysis (\cite{Ma}, p. 384). If $f(\partial S)$ does not self intersect, by the Jordan Curve Theorem \cite{Ha} the winding number $I(f(\partial S), x )$ equals 0 or 1, then the function is injective.
}
\end{proof}

\section*{The construction}

To obtain a net one has to cut a convex polyhedron $P$ along a spanning tree $T$ of the graph of its edges. This way we obtain a flat surface $P_T$, which is homeomorphic to a closed disc. The surface $P_T$ can be mapped isometrically face by face into the plane in an essentially unique way; denote this map by $f_T\colon P_T\to \mathbb{R}^2$.

If $f_T$ is injective then the image $f_T(P_T)$ is a net of $P$. However $f_T$ might not be injective as figure \ref{Tet} shows. Therefore Shephard's problem is asking if for any convex polyhedron $P$ there is a spanning tree $T$ such that $f_T$ is injective.

Since the number of edges of $P$ is finite, we can rotate it so that none of its edges is orthogonal to the $x$ axis. Then, after a suitable stretching in the direction of the $x$ axis, all the edges of the obtained polyhedron $Q$  form a sufficiently small angle with the $x$ axis (less than $\frac{\pi}{20 N}$ will do, where $N$ is the number of edges of $P$).

We now define an ordering on the set of vertices of $Q$. We say that $v \leq v^{\prime}$ if the first coordinate of $v$ is less or equal than the first coordinate of $v^{\prime}$. We will denote by $y$ and $z$ the minimal and maximal vertices respectively.

We define an \textit{ascending sequence} as a set of vertices $\{ p_0, p_1, \ldots , p_n   \}$ such that $p_i \leq p_{i+1}$ and $p_i p_{i+1}$ are connected by an edge for all $i \in \{ 0,1, \ldots, n-1 \} $. We say that a spanning tree $T$ with root $z$ is \textit{increasing} if for any vertex $v\in Q$ there is a (unique) ascending sequence from $v$ to $z$ contained in $T$. 

Note that all terminal edges of an increasing tree form an angle close to $\pi$ with $e_1$. Otherwise, the path in $T$ joining the corresponding leaf to $z$ wouldn't be an ascending sequence (see figure \ref{Mon}). Also, all the ends of the surface $Q \backslash T$ point in the direction of the $x$ axis.

\begin{figure}[h]
\centering
\includegraphics[scale=0.84]{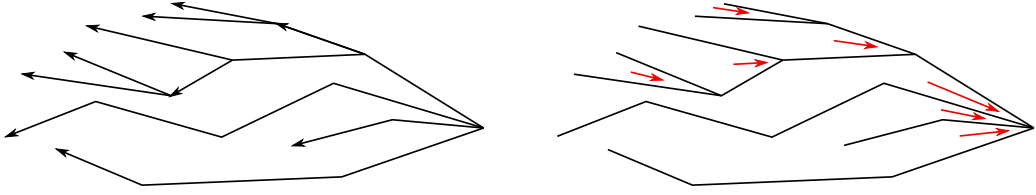}
\caption{Terminal edges of an increasing tree $T$ point leftwards and ends of $Q \backslash T$ point rightwards.}\label{Mon}
\end{figure}

The following theorem clearly implies theorem \ref{Str}.


\begin{teorema}\label{Gho}
{\textbf{(Ghomi)} If $Q$ is cut along any increasing tree $T$, then the unfolding map $f_T$ is injective.}
\end{teorema}

First, choose an unfolding $f_T$ of $Q_T$ in which the images of all the edges are almost horizontal (they form an angle less than $\frac{\pi}{10}$ with $e_1 \in \R ^2$). Next, we are going to prove that $f_T(\partial Q_T)$ does not self intersect.

Consider $y^{\prime} = f_T(y) \in \R ^2$. Then, starting at $y^{\prime}$, the counterclockwise image of $\partial Q_T$ is a sequence of piecewise linear curves (broken lines) going almost horizontally rightwards and leftwards, alternately. We are going to denote the first broken line going right as $R_1$, the first one going left as $L_1$, and so on. Since $T$ is ascending, when we switch from going rightwards to leftwards, we rotate counterclockwise and when we switch from going leftwards to rightwards, we rotate clockwise (see figure \ref{Cro}).

\begin{figure}[h]
\centering
\psfrag{A}{$R_1$}
\psfrag{B}{$L_1$}
\psfrag{C}{$R_2$}
\psfrag{D}{$L_2$}
\psfrag{E}{$R_3$}
\psfrag{F}{$L_3$}
\psfrag{G}{$R_4$}
\includegraphics[scale=0.7]{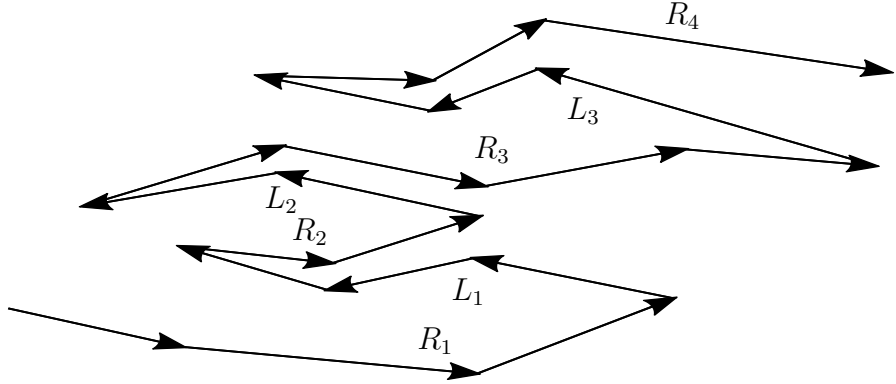}
\caption{The image of the boundary is a sequence of broken lines going rightwards and leftwards alternately.}\label{Cro}
\end{figure}

\begin{prop}\label{Zig}
{As defined above, if the sequence of broken lines $R_1L_1R_2\ldots$ 
$ L_n$ doesn't self intersect, then $R_1L_1 R_2 \ldots L_nR_{n+1}$ doesn't self intersect either.  

}
\end{prop}

\begin{proof}
The proof goes by induction on the number $n$ of times it has gone leftwards (in figure \ref{Cro}, $n=3$). For the case $n=1$, observe that the condition $arg (v_{i} - v_{i-1} )\geq arg (u_{i} - u_{i-1})$ for $i =1, 2, \ldots, m$ in the arm lemma is implied by $arg (v_1 - v_0 )\geq arg (u_1 - u_0)$ and $\angle v_{i+1}v_iv_{i-1}+ \angle u_{i-1}u_iu_{i+1} \leq 2\pi$ for $i\in \{ 1,2, \ldots, m-1\}$. Therefore the arm lemma completes the base of induction.

Suppose the assertion is true for $n \leq k$ and consider the case $n=k+1$.  Note that the edges of $\partial Q_T$ are paired in such a way that we glue paired edges together to obtain $Q$ from $Q_T$. Each one will be called the \textit{dual} of the other.

Observe that when we start $R_2$, we are traveling the dual edges of the leftmost part of $L_1$. If the length of $R_2$ is greater than or equal to the length of $L_1$, we can apply the induction hypothesis to $R_2L_2R_3 \ldots L_nR_{n+1}$ and the result will follow. 

If the length of $R_2$ is less than the length of $L_1$, we extend $R_2 $ with a broken line $S$ parallel to the portion of $L_1$ without the dual of $R_2$ (see figure \ref{S}). By the arm lemma, $S$ will be above $L_1$, and by the induction hypothesis $S$ does not touch $L_2R_3 \ldots L_nR_{n+1}$.
\end{proof}

\begin{figure}[h]
\centering
\includegraphics[scale=0.7]{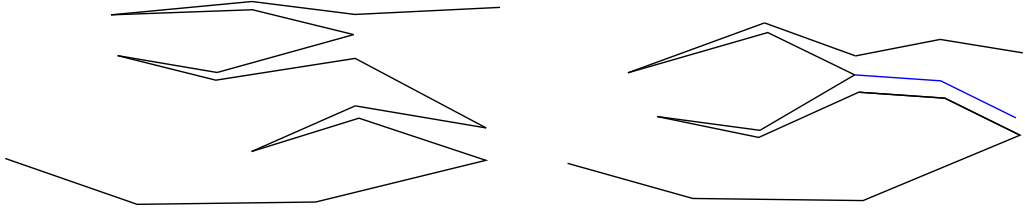}
\caption{On the left $R_2$ and $L_1$ have the same length, on the right we construct $S$ in blue.} \label{S}
\end{figure}

Now, the image of $\partial  Q_T$ will self intersect for the first time in a point $y^{\prime\prime}$ while going leftwards. Because $T$ is increasing, $f_T(\partial Q_T)$ contains a simple closed curve $\gamma$ starting at $y^{\prime \prime}$ consisting of a sequence of broken lines turning clockwise when changing from going leftwards to rightwards and counterclockwise in the other case.

\begin{figure}[h]
\centering
\psfrag{A}{$y^{\prime}$}
\psfrag{B}{$y^{\prime\prime}$}
\includegraphics[scale=0.7]{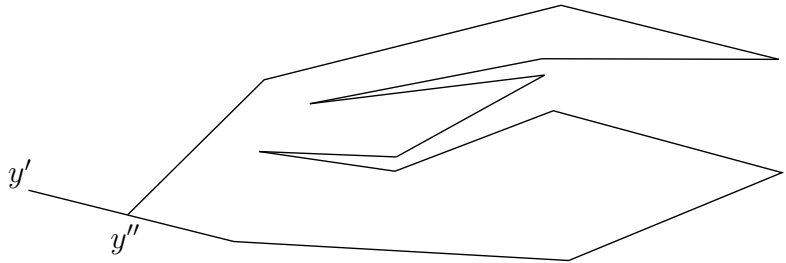}
\caption{The image of $\partial Q_T$ first self intersects at $y^{\prime \prime}$.}
\end{figure}

Since all ends point rightwards, we can contract the curve $\gamma$ to the point  $y^{\prime \prime}$ moving leftwards all the time. Such contraction can be performed in the same way in $Q_T$. Therefore only one point of $\partial Q _T$ is sent to $y^{\prime \prime}$. This is only possible if $y^{\prime} = y^{\prime \prime}$ and $f_T$ restricted to $\partial Q_T$ is injective.

Applying lemma \ref{Argument} with $S=Q_T$ we get that $f_T$ is a net, completing the proof.

As an obvious consequence of the proof we get the following result.

\begin{teorema}\label{Gho}
{\textbf{(Ghomi)} For any convex polyhedron, a sufficiently long stretch in almost any direction yields a polyhedron with a net.}
\end{teorema}

\end{document}